\documentclass[%
 letterpaper,
  twocolumn,
  colorlinks,
]{preprint}

\usepackage[utf8]{inputenc}	
\usepackage[T1]{fontenc}
\usepackage{dsfont}

\usepackage{amsmath}
\usepackage{amsfonts}
\usepackage{amssymb}
\usepackage{amscd}
\usepackage{amsthm}
\usepackage{bbm}

\usepackage[english]{babel}

\usepackage{graphicx}
\usepackage{caption}
\usepackage{subcaption}

\usepackage{tikz}
\usepackage{pgfplots}
\pgfplotsset{compat = 1.17}
\pgfkeys{/pgf/number format/.cd,1000 sep={\,}}

\usepackage{algorithm}
\usepackage{etoolbox}
\let\classAND\AND
\let\AND\relax
\usepackage{algorithmic}

\let\AND\classAND
\AtBeginEnvironment{algorithmic}{\let\AND\algoAND}
\floatname{algorithm}{Algorithm}

\usepackage{xspace}
\usepackage{hyperref}    
\usepackage{nicefrac} 
\usepackage{empheq} 
\usepackage{cancel}

\usepackage{lineno}
\usepackage{epstopdf}
\usepackage{tikz}
\usepackage{pgfgantt}
\usepackage{pdflscape}
\usepackage{algorithm}
\usepackage{hyperref} 
\usepackage[capitalise]{cleveref}
\usepackage{todonotes}

\newcommand\C{\ensuremath{\mathbb{C}}}
\newcommand\R{\ensuremath{\mathbb{R}}}
\newcommand\A{\ensuremath{\mathbb{A}}}
\newcommand\D{\ensuremath{\mathbb{D}}}
\newcommand\hardy{\ensuremath{\mathcal{H}}}
\newcommand\newhardy{\ensuremath{\mathcal{H}_2(\bar{\mathbb{A}}^{\mathsf{c}})}}
\newcommand\re{\ensuremath{\text{Re}}}
\newcommand\im{\ensuremath{\text{Im}}}

\newcommand\iunit{\ensuremath{\mathrm{i}}}
\newcommand\bfI{\ensuremath{\mathbf{I}}}

\newcommand\bfC{\ensuremath{\mathbf{C}}}
\newcommand\bfB{\ensuremath{\mathbf{B}}}
\newcommand\bfA{\ensuremath{\mathbf{A}}}
\newcommand\bfK{\ensuremath{\mathbf{K}}}
\newcommand\bfN{\ensuremath{\mathbf{N}}}
\newcommand\bfV{\ensuremath{\mathbf{V}}}
\newcommand\bfSigma{\ensuremath{\mathbf{\Sigma}}}
\newcommand\ctrlg{\ensuremath{\mathbf{X}_c}}
\newcommand\obsg{\ensuremath{\mathbf{X}_o}}
\newcommand\bfG{\ensuremath{\mathbf{G}}}

\newcommand\mathfrakH{\ensuremath{\mathfrak{H}}}
\newcommand\fro{\ensuremath{\mathsf{F}}}

\newtheorem{theorem}{Theorem} 
\newtheorem{lemma}{Lemma}
\newtheorem{assumption}{Assumption} 
\newtheorem{definition}{Definition}
\newtheorem{remark}{Remark}

\begin{document}

\title{Balanced truncation with conformal maps}

\author[$\dagger$]{Alessandro Borghi}
\affil[$\dagger$]{%
  Technical University of Berlin, Mathematics Department, Stra\ss e des 17. Juni 136, 10623 Berlin, Germany.\authorcr
  \email{borghi@tu-berlin.de}, \orcid{0000-0002-5333-3074} \authorcr \email{tobias.breiten@tu-berlin.de}, \orcid{0000-0002-9815-4897}
}

\author[$\dagger$]{Tobias Breiten}

\author[$\ddagger$]{Serkan Gugercin}
\affil[$\ddagger$]{%
  Department of Mathematics and Division of Computational Modeling and Data
  Analytics, Academy of Data Science, Virginia Tech,
  Blacksburg, VA 24061, USA.\authorcr
  \email{gugercin@vt.edu}, \orcid{0000-0003-4564-5999}
}

\shorttitle{Balanced truncation with conformal maps}
\shortauthor{A. Borghi, T. Breiten, S. Gugercin}
\shortdate{\today}
  


\abstract{
We consider the problem of constructing reduced models for large scale systems with poles in general domains in the complex plane (as opposed to, e.g., the open left-half plane or the open unit disk). 
Our goal is to design a model reduction scheme, building upon theoretically established methodologies, yet  encompassing this new class of models.
To this aim, we develop a balanced truncation framework through conformal maps to handle poles in general domains. The major difference from classical balanced truncation resides in the formulation of the Gramians. We show that these new Gramians can still be computed by solving modified Lyapunov equations for specific conformal maps. 
A numerical algorithm to perform balanced truncation with conformal maps is developed and is tested on three numerical examples, namely a heat model, the Schr\"odinger equation, and the undamped linear wave equation, the latter two having spectra on the imaginary axis. 
}

\novelty{%
}

\maketitle

\section{Introduction} \label{sec:1}
We consider large-scale linear time invariant (LTI) systems of the form 
\begin{equation}\label{eq:fom}
	\begin{cases}
		\dot{\mathbf{x}}(t)=\mathbf{A}\mathbf{x}(t)+\mathbf{B}\mathbf{u}(t),\\
		\mathbf{y}(t) = \mathbf{C} \mathbf{x}(t), \quad \mathbf{x}(0)=0,   
  \end{cases}
\end{equation}
with $\bfA\in\C^{n\times n}$, $\bfB\in\C^{n\times m}$, and $\bfC\in\C^{q\times n}$. In~\eqref{eq:fom},  $\mathbf{x}(t)\in\C^n$, $\mathbf{u}(t)\in\C^m$, and $\mathbf{y}(t)\in\C^q$ denote, respectively, the states, inputs, and outputs of the LTI system. 
Throughout the paper we mainly consider the frequency domain description of \eqref{eq:fom} given by the transfer function 
\begin{equation}\label{eq:tffom}
    \bfG(\cdot) = \bfC(\cdot\bfI-\bfA)^{-1}\bfB.
\end{equation}
The system described by \eqref{eq:fom} and \eqref{eq:tffom} is referred to as the full order model (FOM). 
In the case of a large scale system the computational effort to solve \eqref{eq:fom} for different input signals can often be prohibitive.  The aim of model order reduction is to compute a reduced order model (ROM) that resembles the input output behaviour of \eqref{eq:fom} while drastically lowering the state dimension. 
More specifically, the objective is to determine a surrogate model of \eqref{eq:fom} with the same structure, i.e.,
\begin{equation}\label{eq:rom}
	\begin{cases}
 \dot{\mathbf{x}}_r(t)=\mathbf{A}_r\mathbf{x}_r(t)+\mathbf{B}_r\mathbf{u}(t),\\
		\mathbf{y}_r(t) = \mathbf{C}_r \mathbf{x}_r(t), \quad \mathbf{x}_r(0)=0,   
  \end{cases}
\end{equation}
and transfer function  
\begin{equation}\label{eq:tfrom}
    \bfG_r(\cdot)=\bfC_r(\cdot\bfI-\bfA_r)^{-1}\bfB_r,
\end{equation}
where $\mathbf{A}_r\in\mathbb{C}^{r\times r}$, $\bfB_r\in\mathbb{C}^{r\times m}$, and  $\bfC_r\in\mathbb{C}^{q\times r}$, such that the output behaviour $\mathbf{y}_r$ well approximates $\mathbf{y}$ for a set of inputs $\mathbf{u}$. In particular, for the model to be computationally efficient, we impose $r\ll n$. As a metric of disparity between the two models $\bfG$ and $\bfG_r$, generally the $\hardy_\infty$ or the $\hardy_2$ norms are  used (see, e.g., \cite[Section 5.1.3]{Ant05}). 

Many model order reduction techniques have been developed to approximate the systems of the 
form~\eqref{eq:fom}. We refer the reader to \cite{Ant05,AnBeGu20,BeMeSo05,BenOCW17,DeG21,QuaMN16,Volk13} and the extensive references therein for a 
detailed overview of different techniques. The framework developed in this article is closely related to balanced truncation (BT)~\cite{Moo81,MuRo76}, one of the gold standards in system theoretic approaches to model reduction, and its extension to structured differential equations \cite{SorAnt05,Bre16}.
In this paper, we focus on the classical Lyapunov balancing; for details on the other variants of BT, we refer the reader to the  survey articles \cite{GuAn04,BreS21}. Furthermore, here we focus on the projection-based formulation of BT. For a data-driven formulation of BT
using only transfer function  
evaluations, see  \cite{GoGuBe22}.

In this paper, we assume that the poles of $\bfG$, i.e., the eigenvalues of $\bfA$, lie in $\A\subset\C$, a non-empty connected open set (which is not necessarily the open left-half plane as usually assumed). We then  adopt the conformal mapping framework introduced in \cite{BorBre23} to extend BT to LTI systems with poles in general domains $\A\subset\C$. More specifically, the main contributions are the following:
\begin{enumerate}
    \item Via conformal mappings, we develop the Gramians of an LTI system with poles in general domains and, consequently, we extend the balanced truncation algorithm to this class of systems.
    \item We prove that, for some choice of conformal mappings, the Gramians are the solutions of modified Lyapunov equations.
    \item We prove that the resulting reduced model 
    preserves stability when specific conformal maps are adopted. In addition, we provide an a-posteriori bound on an appropriately modified $\hardy_2$ like norm.
    \item We develop an algorithmic framework and show the effectiveness of the proposed method on a diverse set of examples with poles in different domains.
\end{enumerate}

The structure of the paper is as follows. In \cref{sec:preliminaries} we review some basic facts on  conformal maps and balanced truncation. 
\cref{sec:3} introduces our main result, a conformal mapping framework for BT, and the corresponding algorithm.
In \cref{sec:4} we discuss some theoretical results on BT with conformal maps. Specifically, we prove stability preservation of the reduced model when specific conformal maps are used, and develop a bound on the $\hardy_2$ error norm.  
Three numerical experiments with the proposed algorithm are provided in \cref{sec:5}. Here, we use partial differential equations with spectra on the left-half complex plane and on the imaginary axis.

\subsection{Notation} Throughout the paper we indicate with $\|\cdot\|_{\fro}$ the Frobenius norm and $\|\cdot\|_2$ the spectral norm. 
The absolute value of a complex number $z$ is denoted by $|z|=\sqrt{zz^*}$. 
The symbol $\iunit$ indicates the imaginary unit. The symbol $(\cdot)^*$ indicates the complex conjugation of a scalar or the conjugate transpose of a matrix. 
If $\A$ is an open subset of the complex plane, $\partial\A$ denotes its boundary, $\bar{\A}=\{\A\cup\partial\A\}$ its closure, $\A^{\mathsf{c}}$  its complement, and $\bar{\A}^{\mathsf{c}}=\{\C\backslash\bar{\A}\}$ its exterior. 
The symbols $\C$, $\C_-$, and $\C_+$ stand for the complex plane, the open left-half complex plane, and the open right-half complex plane, respectively. In addition, $\R$ and $\iunit\R$ indicate the real numbers and the imaginary numbers, respectively. 
For a single-variable complex-valued differentiable bijective function $f$ we indicate its complex derivative by $f'$ and its inverse by $f^{-1}$.
In addition, for a complex-valued function $g$, we indicate the composition of $f$ and $g$ as $f\circ g$ or $f(g(\cdot))$. 
For the numerical examples, spatially localized controls on intervals $[a,b]$ are addressed with the indicator function $\chi_{[a,b]}$. 

\section{Preliminaries} \label{sec:preliminaries}

\subsection{A conformal mapping framework} \label{sec:2}
As stated in \cref{sec:1}, this paper considers LTI systems with poles in general domains $\A$ that are not necessarily the unit disk nor the left-half plane. For the FOM with transfer function~\eqref{eq:tffom}, to simplify the presentation, we assume that $\bfA$ has the eigendecomposition $\bfA = \bfV\mathbf{\Lambda}\bfV^{-1}$, $\mathbf{\Lambda}=\textnormal{diag}(\lambda_1,\dots,\lambda_n)$, with simple eigenvalues.
As we discuss in \cref{remark:schur} the analysis can be extended to the general case via  Schur decomposition.
For $\A\subset\C$ being a non-empty connected open set, we then have that $\lambda_j\in\A$, for $j=1,\dots,n$. 
To adapt balanced truncation to this type of systems, we develop a new framework that relies on conformal maps. {For this reason, we recall the conformal mapping theorem below.} 
\begin{theorem}[\cite{Weg12}, Theorem 6.1.2]\label{th:com} 
	Suppose $\mathbb{X},\mathbb{Y}\subset\mathbb{C}$ are open sets and let $\psi\colon\mathbb{X}\rightarrow\mathbb{Y}$ be Fr\'echet differentiable as a function of two real variables. The mapping $\psi$ is conformal in $\mathbb{X}$ if and only if it is analytic in $\mathbb{X}$ and $\psi'(z_0)\neq 0$ for every $z_0\in\mathbb{X}$.
\end{theorem}
Throughout the paper, we make the following assumptions
\begin{assumption}\label{assumption:1}
    We assume that 
    \begin{itemize}
        \item [(2.1)] $\psi\colon\mathbb{X}\rightarrow\A$ is a bijective conformal map where $\mathbb{X}\subseteq\C_-$ such that its boundary $\partial\mathbb{X}$ includes the imaginary axis $\iunit\R$. 
        \item [(2.2)] $\psi\colon\tilde{\mathbb{X}}\rightarrow\bar{\mathbb{A}}^{\mathsf{c}}$ is also conformal with $\tilde{\mathbb{X}}\subseteq\C\backslash\{\mathbb{X}\cup\iunit\R\}$.
        \item [(2.3)] Let $\partial\A^+$ the boundary of $\A$ such that its interior includes the eigenvalues of $\bfA$, we then consider $\psi\circ\iunit\colon\R\rightarrow\partial\A^+$ to be continuously differentiable and $\psi'(z)\neq0$ for every $z\in\R$.
    \end{itemize}
\end{assumption}
A simplified graphical depiction of \cref{assumption:1} is given in \cref{fig:confmap} (a more involved example is given in \cref{fig:joukowskiwave}). Note that, by the inverse mapping theorem, \cref{assumption:1}.3 guarantees bijectivity of $\psi$ on the imaginary axis. Some of the main results in this paper, specifically \cref{th:mobiuslyap} and \cref{th:stabilitypreservation}, use a M\"obius transformation $m$ as conformal map satisfying \cref{assumption:1}. 
The M\"obius transformation $m$ and its inverse are given by 
\begin{equation}\label{eq:mobiustrans}
\begin{aligned}
m(\cdot) &= \frac{\alpha \cdot+\beta}{\gamma \cdot + \delta}, \quad m^{-1}(\cdot) = \frac{\beta-\delta\cdot}{\gamma\cdot - \alpha},\\ &\textnormal{with } \alpha,\beta,\gamma,\delta\in\mathbb{C}\; \textnormal{ and }\;\alpha\delta -\beta\gamma\neq0.
\end{aligned}
\end{equation}
We refer the reader to \cite[Section 6.3]{Weg12} for more details. 
In addition, we also apply the M\"obius transformation to matrices. Consider the matrix $\bfA$, we then have the following definitions: 
\begin{equation}\label{eq:mobiustransmatrix}
\begin{aligned}
m(\bfA) &= (\alpha \bfA+\beta\bfI)(\gamma \bfA+ \delta)^{-1},\\ 
m^{-1}(\bfA) &= (\beta\bfI-\delta\bfA)(\gamma\bfA - \alpha\bfI)^{-1}, 
\end{aligned}
\end{equation}
with the parameters $\alpha,\beta,\gamma,\delta$ being as in \eqref{eq:mobiustrans}. In this manuscript, the application of a scalar function $m$ to a matrix $\bfA$ follows the definition given in \cite{Hig08}. More in detail, for the eigendecomposition of $\bfA$, we define
\[
m(\bfA) := \bfV m(\mathbf{\Lambda}) \bfV^{-1} = \bfV \begin{bmatrix}
    m(\lambda_1) & & \\
    & \ddots & \\
    & &  m(\lambda_n)
\end{bmatrix}\bfV^{-1}.
\]
A similar definition applies for $m^{-1}$.

In the next section, we introduce a space of square integrable functions with poles in general domains. 

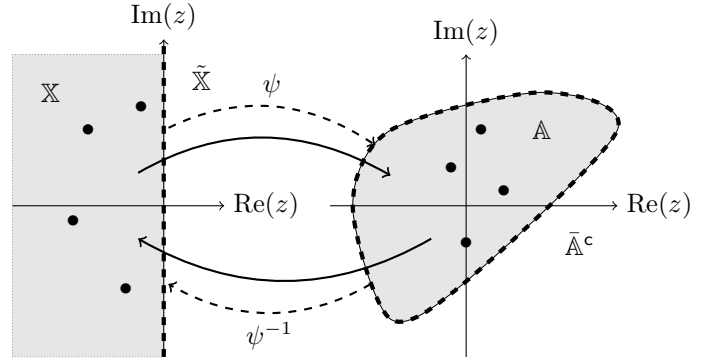
\begin{figure}
    \centering
    \begin{tikzpicture}[scale=1]
        \draw[ultra thin, densely dotted, fill=gray!20] (-3,-2) rectangle (-1,2);
        \draw[->] (-3,0) -- (-0.2,0) node[right] {$\text{Re}(z)$};
        \draw[->] (-1,-2) -- (-1,2.2) node[above] {$\text{Im}(z)$};
        \draw[ultra thick, dashed , black] (-1,-2) -- (-1,2.2);
        \node at (-2.5,1.5) {$\mathbb{X}$};
        \node at (-0.5,1.7) {$\tilde{\mathbb{X}}$};
        \draw (-2,1) node {$\bullet$};
        \draw (-2.2,-0.2) node {$\bullet$};
        \draw (-1.5,-1.1) node {$\bullet$};
        \draw (-1.3,1.3) node {$\bullet$};
        
        \draw[ultra thin, fill=gray!20] plot [smooth cycle] coordinates {(3.5-2,0) (4-2,1) (6-2,1.5) (7-2,1) (5-2, -1) (4-2,-1.5)};
        \draw [black, dashed, ultra thick] plot [smooth cycle] coordinates {(3.5-2,0) (4-2,1) (6-2,1.5) (7-2,1) (5-2, -1) (4-2,-1.5)};
        \draw[->] (3.2-2,0) -- (7-2,0) node[right] {$\text{Re}(z)$};
        \draw[->] (5-2,-2) -- (5-2,2) node[above] {$\text{Im}(z)$};
        \node at (6-2,1) {$\mathbb{A}$};
        \node at (6.5-2,-0.5) {$\bar{\mathbb{A}}^{\mathsf{c}}$};
        \draw (5.2-2,1) node {$\bullet$};
        \draw (5.5-2,0.2) node {$\bullet$};
        \draw (4.8-2,0.5) node {$\bullet$};
        \draw (5-2,-0.5) node {$\bullet$};
        
        \draw[->, thick, shorten >=0pt, shorten <=2pt, dashed] (-1,1) to[bend left] node[midway, above] {$\psi$} (3.8-2,0.8);
        \draw[->, thick, shorten >=0pt, shorten <=2pt] (-1.4,0.4) to[bend left] node[midway, above] {} (4-2,0.4);
        \draw[->, thick, shorten >=2pt, shorten <=2pt] (4.6-2,-0.4) to[bend left] node[midway, below] {} (-1.4,-0.4);
        \draw[->, thick, shorten >=2pt, shorten <=2pt, dashed] (3.8-2,-1) to[bend left] node[midway, below] {$\psi^{-1}$} (-1,-1);
        
    \end{tikzpicture}
    \caption{An illustration of a conformal map satisfying \cref{assumption:1}. The arrows between the grey sets $\mathbb{X}$ and $\mathbb{A}$ indicate the bijectivity of $\psi$. The same holds for the dashed arrow lines between the dashed boundaries $\iunit\R$ and $\partial\A^+$ (in this depiction $\partial\A^+$ coincides with $\partial\A$). The dots $\bullet$ indicate the poles of the transfer function $\bfG$. In addition, $\psi$ conformally maps the white sets $\tilde{\mathbb{X}}$ and $\bar{\mathbb{A}}^{\mathsf{c}}$.}
    \label{fig:confmap}
\end{figure}

\subsection{The \texorpdfstring{$\newhardy$}{H\_2(barA\^c)} space}
In model order reduction, $\hardy_2$ denotes a particular Hardy space. More precisely, it denotes the Hilbert space consisting of all functions $\mathbf{F}$ and $\mathbf{H}$, respectively, analytic in $\C_+$ satisfying
\[
\sup_{x>0}\int_{-\infty}^\infty \|\mathbf{F}(x+\iunit\omega)\|^2_{\fro}\,\mathrm{d}\omega<\infty
\]
with the inner product
\begin{equation*}
    \langle \mathbf{F},\mathbf{H}\rangle_{\hardy_2}:= \frac{1}{2\pi}\int_{-\infty}^\infty \textnormal{trace}\left\{\mathbf{F}(\iunit\omega)\mathbf{H}(\iunit\omega)^*\right\} \mathrm{d}\omega,
\end{equation*}
and the corresponding norm 
\begin{equation*}
    \|\mathbf{F}\|_{\hardy_2}:= \left(\frac{1}{2\pi}\int_{-\infty}^\infty \|\mathbf{F}(\iunit\omega)\|^2_\fro\, \mathrm{d}\omega\right)^{\frac{1}{2}}.
\end{equation*}

In the cases studied in this paper, $\mathbf{F}$ is analytic in $\bar{\A}^{\mathsf{c}}$. Here, $\bar{\A}^{\mathsf{c}}$ does not necessarily have to be the open right half complex plane, meaning that $\mathbf{F}$ is not necessarily in $\hardy_2$. 
Due to this obstacle, we adopt the framework introduced in \cite{BorBre23} (see also \cite[Chapter 10]{Dur70}). Here, the classical $\mathcal{H}_2$ space is replaced by the space consisting of all the functions $\mathbf{F}$ for which $(\mathbf{F}\circ\psi(\cdot))\psi'(\cdot)^{\frac{1}{2}}\in\hardy_2$, where $\psi$ is a given conformal map. 
We generalize the definition of this space given in \cite{BorBre23} for matrix valued functions. 
\begin{definition}[$\mathcal{H}_2(\bar{\mathbb{A}}^{\mathsf{c}})$ space, \cite{BorBre23}] \label{def:H2A}
	 Let $ \mathbf{F}\colon\bar{\mathbb{A}}^{\mathsf{c}}\rightarrow\mathbb{C}^{q\times m}$ and $  \mathbf{H}\colon\bar{\mathbb{A}}^{\mathsf{c}}\rightarrow\mathbb{C}^{q\times m}$ be analytic. 
	 Define
	 \begin{equation} \label{eq:operatorH}
	 	\mathfrak{H}_\mathbf{F}(\cdot) = \left(\mathbf{F}\circ\psi(\cdot)\right)\psi'(\cdot)^{\frac{1}{2}}.
	 \end{equation} 
	 Then the $\mathcal{H}_2(\bar{\mathbb{A}}^{\mathsf{c}})$ inner product is defined as 
	 \begin{equation*}
	 	\left\langle  \mathbf{F},  \mathbf{H}\right\rangle_{\mathcal{H}_2(\bar{\mathbb{A}}^{\mathsf{c}})} := \left\langle \mathfrak{H}_\mathbf{F},\mathfrak{H}_\mathbf{H}\right\rangle_{\mathcal{H}_2}
	 \end{equation*}
	with the  
    corresponding $\mathcal{H}_2(\bar{\mathbb{A}}^{\mathsf{c}})$-norm 
	\begin{equation*}
		\| \mathbf{F}\|_{\mathcal{H}_2(\bar{\mathbb{A}}^{\mathsf{c}})} := \left\|\mathfrak{H}_\mathbf{F}\right\|_{\mathcal{H}_2}=\left( \left\langle \mathfrak{H}_\mathbf{F},\mathfrak{H}_\mathbf{F}\right\rangle_{\mathcal{H}_2}\right)^{\frac{1}{2}}.
	\end{equation*}
	The space $\mathcal{H}_2(\bar{\mathbb{A}}^{\mathsf{c}})$  is defined as 
	\begin{equation*}
		\mathcal{H}_2(\bar{\mathbb{A}}^{\mathsf{c}}):=\left\{ \mathbf{F}\colon\bar{\mathbb{A}}^{\mathsf{c}}\rightarrow\mathbb{C}^{q\times m}\;\textnormal{analytic}\; \bigg\lvert \| \mathbf{F}\|_{\mathcal{H}_2(\bar{\mathbb{A}}^{\mathsf{c}})}<\infty\right\}.
	\end{equation*}
\end{definition}
\Cref{def:H2A} implies that if $\mathbf{F}\in\newhardy$ then $\mathfrak{H}_\mathbf{F}\in\hardy_2$. 
Given the particular structure of an LTI system's transfer function $\bfG$ as in \eqref{eq:tffom}, we write the corresponding operator $\mathfrak{H}_\bfG$ as 
\begin{equation}\label{eq:frakH}
\begin{aligned}
    \mathfrak{H}_\bfG(\cdot) &= (\bfG\circ\psi(\cdot))\psi'(\cdot)^{\frac{1}{2}} \\ &= \bfC\left( \psi(\cdot)\bfI - \bfA\right)^{-1}\bfB\psi'(\cdot)^{\frac{1}{2}}\\
    &= \bfC\bfK(\cdot)^{-1}\bfB,
\end{aligned}
\end{equation}
where 
\begin{equation} \label{eq:K}
\bfK(\cdot) = \psi(\cdot)\psi'(\cdot)^{-\frac{1}{2}}\bfI - \bfA\psi'(\cdot)^{-\frac{1}{2}}.
\end{equation}

\subsection{Balanced truncation} \label{sec:classicBT}
In this section, we briefly review the concept of balanced truncation (BT) for asymptotically stable LTI systems with poles in $\A = \C_-$. 
BT is a (Petrov-Galerkin) projection-based model reduction technique. In other words, it constructs two model reduction bases $\bfV_r \in \C^{n\times r} $ and $\mathbf{W}_r \in  \C^{n\times r} $ such that the state-space representation (system matrices) of the ROM in \eqref{eq:rom} is given by
\begin{equation}\label{eq:rommatrices}
    \begin{aligned}
    \bfA_r=\mathbf{W}_r^*\bfA\bfV_r,\; \bfB_r=\mathbf{W}_r^*\bfB,\;\bfC_r =  \bfC\bfV_r.
    \end{aligned}
\end{equation}
BT chooses $\mathbf{V}_r$ and $\mathbf{W}_r$ to eliminate 
hard-to-reach and hard-to-observe states of the original FOM in \eqref{eq:fom} \cite[Section 7.1]{Ant05}. The computation of $\mathbf{V}_r$ and $\mathbf{W}_r$ depends on the controllability and observability Gramians of \eqref{eq:fom}, denoted by $\ctrlg$ and $\obsg$, respectively. For a minimal system, these Gramians are the symmetric {positive definite} unique solutions to the Lyapunov equations 
\begin{equation}\label{eq:classicLyap}
    \bfA\ctrlg + \ctrlg\bfA^* = -\bfB\bfB^*, \quad \bfA^*\obsg + \obsg\bfA = -\bfC^*\bfC.
\end{equation}
The Gramians $\ctrlg$ and $\obsg$ given as the solutions to the Lyapunov equations~\eqref{eq:classicLyap} can be equivalently defined as integrals in the frequency domain, namely
\begin{align}
    \ctrlg &= \frac{1}{2\pi}\int_{-\infty}^{\infty}(\iunit\omega\mathbf{I}-\mathbf{A}  )^{-1}\bfB\bfB^*(\iunit\omega\bfI-\mathbf{A})^{-*} \mathrm{d}\omega, \label{eq:classicPfreq} \\
    \obsg &= \frac{1}{2\pi}\int_{-\infty}^{\infty}(\iunit\omega\bfI-\mathbf{A})^{-*}\bfC^*\bfC(\iunit\omega\bfI-\mathbf{A})^{-1}\mathrm{d}\omega. \label{eq:classicQfreq} 
\end{align}
These frequency domain definitions will play a crucial role in our development of BT via conformal maps in~\Cref{sec:3}.

In practice, one does not solve~\eqref{eq:classicLyap} for $\ctrlg$ and $\obsg$. Instead one solves for their square-root factors. More precisely, 
let $\ctrlg=\mathbf{U}\mathbf{U}^*$ and $\obsg = \mathbf{L}\mathbf{L}^*$ be Cholesky decompositions. The existence of $\mathbf{U}$ and $\mathbf{L}$ is guaranteed via the positive definiteness of $\ctrlg$ and $\obsg$. Then,
one solves~\eqref{eq:classicLyap} directly for $\mathbf{U}$ and $\mathbf{L}$. We refer the reader to \cite{BreS21} for details.  
Let 
\[
\mathbf{U}^*\mathbf{L}=\mathbf{Z}\bfSigma\mathbf{Y}^*=\begin{bmatrix}
    \mathbf{Z}_r & \mathbf{Z}_2
\end{bmatrix}
\begin{bmatrix}
    \mathbf{\Sigma}_1  & \\ & \mathbf{\Sigma}_2 
\end{bmatrix}
\begin{bmatrix}
    \mathbf{Y}_r^* \\ \mathbf{Y}_2^*
\end{bmatrix}, 
\]
be the singular value decomposition where the entries of $\bfSigma$ are called the Hankel singular values of the FOM. Here, $\mathbf{\Sigma}_1 \in\R^{r\times r}$ contains the dominant $r$ singular values, and $\mathbf{Z}_r\in\C^{n\times r}$ and $\mathbf{Y}_r\in\C^{n\times r}$ are corresponding left and right singular vectors. BT then constructs the model reduction matrices as $\mathbf{W}_r=\mathbf{L}\mathbf{Y}_r\mathbf{S}_1^{-1/2}$ and $\mathbf{V}_r=\mathbf{U}\mathbf{Z}_r\mathbf{S}_1^{-1/2}$, which are then used to construct the matrices in \eqref{eq:rommatrices} for the reduced system in \eqref{eq:rom}.

ROM~\eqref{eq:rom} via BT has important advantages. Firstly, $\bfG_r$ in \eqref{eq:tfrom} is asymptotically stable and $\|\bfG-\bfG_r\|_{\hardy_\infty}\leq2\textnormal{trace}\{\mathbf{\Sigma}_2\}$, where  $\|\mathbf{F}\|_{\hardy_\infty}\colon=\sup_{\omega\in\R}\|\mathbf{F}(\iunit\omega)\|_{2}$
denotes the $\hardy_\infty$ norm (see also \cite[Theorem 7.9]{Ant05}). Secondly, there also exists a bound on the $\hardy_2$ error norm~\cite[Section 7.2.2]{Ant05}. We give a brief summary below. Consider the \textit{balanced} realization of $\bfG$ as $\bfG(\cdot) = \bfC_\mathcal{B}(\cdot\bfI - \bfA_{\mathcal{B}})^{-1}\bfB_\mathcal{B}$, i.e., a state-space realization of $\bfG$ such that
\[
\ctrlg=\obsg=\textnormal{diag}(\sigma_1\dots\sigma_n)=\begin{bmatrix}
    \mathbf{\Sigma}_1 & \\ & \mathbf{\Sigma}_2
\end{bmatrix}.
\]
Let $\mathbf{A}_\mathcal{B}$ and $\mathbf{B}_\mathcal{B}$ be partitioned accordingly as 
\begin{equation*}
    \mathbf{A}_\mathcal{B} = \begin{bmatrix}
        \bfA_{11} & \bfA_{12}\\ \bfA_{21} & \bfA_{22}
    \end{bmatrix},\quad \mathbf{B}_\mathcal{B} = \begin{bmatrix}
        \bfB_{1} \\ \bfB_{2}
    \end{bmatrix},
\end{equation*}
and let 
\[
\bfG_2(\cdot) = \bfA_{12}\bfSigma_2\left(\cdot\bfI - \bfA_{22} - \bfA_{21}(\cdot\bfI-\bfA_{11})^{-1}\bfA_{12}\right)^{-1}\bfSigma_2\bfA_{21}.
\]
Then it holds 
\[
\|\bfG-\bfG_r\|_{\hardy_2}^2\leq \textnormal{trace}\{\bfC_2\bfSigma_2\bfC_2^*\} + 2\kappa\|\mathbf{G}_2\|_{\hardy_\infty},
\]
for some $\kappa\in\R$. 
In \cref{th:H2errorbound}  below, one of our main results, we utilize ideas from \cite{SorAnt05} for structured systems to derive an analogous bound for our BT via conformal mapping framework as well.

\section{Balanced truncation with conformal maps}  \label{sec:3}
We now consider $\bfG\in\newhardy$ structured as in \eqref{eq:tffom} (with poles in 
$\A\subset\C$) and extend the concept of BT to these systems. Since the Gramians are the main ingredient of BT, we first need to define them for  $\bfG\in\newhardy$.  

\subsection{Defining the Gramians}
For $\bfG\in\newhardy$, we recall that, for the conformally mapped function $\mathfrak{H}_\bfG$ 
with its \emph{state-space} representation as in~\eqref{eq:frakH}, it holds that $\mathfrak{H}_\bfG\in\hardy_2$. Then, 
inspired by the frequency domain representation of the Gramians in~\eqref{eq:classicPfreq}
and~\eqref{eq:classicQfreq} for the classical case of $\bfG(\cdot) = \mathbf{C}(\cdot \mathbf{I}-\mathbf{A})^{-1}\mathbf{B}$ and by the Gramians defined for integro-differential equations in~\cite{Bre16}, 
we define the controllability and observability Gramians with respect to $\mathfrak{H}_\bfG(\cdot) = \bfC\bfK(\cdot)^{-1}\bfB$ as
\begin{align}
    \ctrlg &= \frac{1}{2\pi}\int_{-\infty}^{\infty}\bfK(\iunit\omega)^{-1}\bfB\bfB^*\bfK(\iunit\omega)^{-*} \mathrm{d}\omega, \label{eq:ctrlgint}\\
    \obsg &= \frac{1}{2\pi}\int_{-\infty}^{\infty}\bfK(\iunit\omega)^{-*}\bfC^*\bfC\bfK(\iunit\omega)^{-1}\mathrm{d}\omega.\label{eq:obsgint}
\end{align}
where $\mathbf{K}(\cdot)$ is as defined in~\eqref{eq:K}.

Unlike the classical BT case for asymptotically stable LTI systems where the Gramians 
can be computed as solutions to the Lyapunov equations~\eqref{eq:classicLyap},  the newly defined Gramians in~\eqref{eq:ctrlgint} and~\eqref{eq:obsgint} cannot be obtained easily
for general conformal maps $\psi$.  
In these general cases, one can compute an approximation of
$\ctrlg$ and $\obsg$ through numerical quadrature. For example, the approximate controllability Gramian can be computed as
\begin{align}
    \ctrlg \approx \tilde{\mathbf{X}}_c &= \sum_{j=1}^{N}w_j\bfK(\iunit p_j)^{-1}\bfB\bfB^*\bfK(\iunit p_j)^{-*}, \label{eq:gramquad}
\end{align}
where $w_j$ and $p_j$ are the quadrature weights and nodes respectively. The observability
Gramian can be approximated similarly.
Even though the quadrature-based approximation \eqref{eq:gramquad} to the Gramians will be employed for general conformal maps, we will show in the next section, more specifically in \cref{th:mobiuslyap}, that for a particular type of conformal mapping, namely the M\"obius transformation, 
the Gramians in \eqref{eq:ctrlgint} and \eqref{eq:obsgint} 
 can still be computed by solving a modified Lyapunov equation.

\subsection{Lyapunov equations}
We start the section with a result on the uniqueness of the solution to a specific Lyapunov-like equation.

\begin{lemma}[Unique solution]\label{lemma:uniqueness} Consider the domains $\A\subset\C$ and $\mathbb{X}\subseteq\C_-$. Let the matrix $\bfA\in\C^{n\times n}$ have the
eigendecomposition {$\bfA=\bfV\mathbf{\Lambda}\bfV^{-1}$}, with $\mathbf{\Lambda}=\textnormal{diag}(\lambda_1,\dots,\lambda_n)$, and $\lambda_j\in\A$, for $j=1,\dots,n$, are distinct. Let $f\colon\A\rightarrow\mathbb{X}$ be analytic.
Then the Lyapunov equation
\begin{equation}\label{eq:lyapunique}
    f(\bfA)\mathbf{P} + \mathbf{P}f(\bfA)^* = \mathbf{Q},
\end{equation}
with $\mathbf{Q} = \mathbf{Q}^*  \in\C^{n\times n}$ , has a unique solution $\mathbf{P}$.
\end{lemma}
\begin{proof}
The main idea behind this result is the fact that the Sylvester equation
\[
\mathbf{Z}\mathbf{P} + \mathbf{P}\mathbf{Y} = \mathbf{Q},
\]
where $\mathbf{Z} \in \C^{n_Z \times n_Z}$,
$\mathbf{Y} \in \C^{n_Y \times n_Y}$, 
$\mathbf{Q} \in \C^{n_Z \times n_Y}$, 
has a unique solution $\mathbf{P}\in \C^{n_Z \times n_Y} $ if and only if  $\mathbf{Z}$ and $-\mathbf{Y}$ do not share any eigenvalues (see \cite[Proposition 6.2]{Ant05}). In our case we have $\mathbf{Z}=f(\bfA)$ and $\mathbf{Y} = f(\bfA)^*$. Given the eigendecomposition of $\bfA$, we can then write 
\begin{equation} \label{eq:fA}
f(\bfA) = \bfV f(\mathbf{\Lambda})\bfV^{-1}=\bfV \begin{bmatrix}
    f(\lambda_1) & & \\
    & \ddots & \\
    & &  f(\lambda_n)
\end{bmatrix}\bfV^{-1},
\end{equation}
(see also \cite[Definition 1.2]{Hig08}). 
Since $\mathbf{\Lambda}\in\A$  we have that $f(\mathbf{\Lambda})\in\mathbb{X}$ with $\mathbb{X}\subseteq\C_-$. Given that the eigenvalues of $f(\bfA)$ are the mirror images of the eigenvalues of $-f(\bfA)^*$ with respect to the imaginary axis, 
then, the two matrices $f(\bfA)$ and $-f(\bfA)^*$ do not share any eigenvalues and thus~\eqref{eq:lyapunique} has a unique solution.
\end{proof}
\begin{remark}\label{remark:schur}
One can prove \cref{lemma:uniqueness} 
without the diagonalizability (and the simple eigenvalues) assumption on $\bfA$. It is avoided here since the notation and presentation becomes rather cumbersome. Next we briefly explain how the argument goes in that case. 
Let $\bfA=\mathbf{U}\mathbf{T}\mathbf{U}^*$ 
be the Schur decomposition  of $\bfA$
where $\mathbf{U}$ is a unitary matrix and $\mathbf{T}$ is an upper triangular matrix. The computation of the matrix function $f(\bfA)$ in~\eqref{eq:fA} can then be carried out, e.g., following the approach discussed in \cite{DavHi03}. Applying \cite[Algorithm 5.1]{DavHi03} results in $f(\bfA)=\mathbf{F}=\mathbf{U}\mathbf{N}\mathbf{U}^*$ where $\mathbf{N}$ is an upper triangular matrix with $f(\lambda_i)$, $i=1,\dots,n$ as the diagonal entries. Since $\mathbf{N}$ and $\mathbf{F}$ are similar, the eigenvalues of $\mathbf{F}$ are given by $f(\lambda_i)$ for $i=1,\dots,n$ which, then, would allow to apply similar arguments as in the  proof of \cref{lemma:uniqueness}. 
\end{remark}

We now show that the newly defined Gramians in \eqref{eq:ctrlgint} and \eqref{eq:obsgint} for 
$\mathfrak{H}_\bfG$ in~\eqref{eq:frakH} 
solve (modified) Lyapunov equations of the form~\eqref{eq:lyapunique} when the M\"obius transformation is adopted in the mapping from $\bfG$ to $\mathfrak{H}_\bfG$.
\begin{theorem}\label{th:mobiuslyap} Consider the transfer function $\bfG\in\newhardy$ with poles $\lambda_j\in\A, \; j=1,\dots,n$, and the M\"obius transformation $m(\cdot) = (\alpha\cdot + \beta)/(\gamma\cdot + \delta)$ in \eqref{eq:mobiustrans} such that $m\colon \C_-\rightarrow \A$ and $m\colon \iunit\R\rightarrow\partial\A^+$. Then the controllability and observability Gramians $\ctrlg,\obsg$ in~\eqref{eq:ctrlgint} and~\eqref{eq:obsgint}
are the unique solutions of the Lyapunov equations
\begin{align}
m^{-1}(\bfA)\ctrlg + \ctrlg m^{-1}(\bfA)^* &= -\mathbf{Q}_c,~\mbox{and} \label{eq:lyapcon}\\
\obsg m^{-1}(\bfA) + m^{-1}(\bfA)^*\obsg &= -\mathbf{Q}_o, \label{eq:lyapobs}
\end{align}
where $m^{-1}(\bfA)$ is defined as in \eqref{eq:mobiustransmatrix}, and
\begin{align*}
    \mathbf{Q}_c&=|\alpha\delta-\beta\gamma|(\alpha\bfI- \gamma\bfA)^{-1}\mathbf{B}\mathbf{B}^*(\alpha\bfI- \gamma\bfA)^{-*},  \\
    \mathbf{Q}_o&=|\alpha\delta-\beta\gamma|(\alpha\bfI- \gamma\bfA)^{-*}\mathbf{C}^* \mathbf{C}(\alpha\bfI- \gamma\bfA)^{-1}. \\
\end{align*}
\end{theorem}
\begin{proof}
We first focus on the controllability Gramian. For the specific mapping $m$, the Gramian in \eqref{eq:ctrlgint} is given by
\begin{align*}
\ctrlg
&= \frac{1}{2\pi}\int_{-\infty}^{\infty}(m(\mathrm{i}\omega)\bfI-\bfA)^{-1}\mathbf{B}\left(m'(\mathrm{i}\omega)^{1/2}\right) \\
&\hspace{2.5cm}\left(m'(\mathrm{i}\omega)^{1/2}\right)^*\mathbf{B}^*\left(m(\mathrm{i}\omega)\mathbf{I}-\mathbf{A}\right)^{-*}\mathrm{d}\omega,
\end{align*}
where 
\[
m'(\cdot) = \frac{\alpha\delta - \beta\gamma}{(\gamma \cdot +\delta)^2},
\]
is the derivative of $m$.
We then obtain 
\begin{align} \nonumber
	&\ctrlg =\\
    \; &= \frac{1}{2\pi}\int_{-\infty}^{\infty}\left(m(\iunit\omega)\bfI-\bfA\right)^{-1}\mathbf{B}\frac{|\alpha\delta-\beta\gamma|}{(\gamma \iunit\omega+\delta)(\gamma \iunit\omega+\delta)^*} \nonumber \\
    & \hspace{5cm}\mathbf{B}^*\left(m(\iunit\omega)\bfI-\bfA\right)^{-*}\mathrm{d}\omega \nonumber \\	
	\; &= \frac{1}{2\pi}\int_{-\infty}^{\infty}\mathbf{R}(\iunit\omega)^{-1}\mathbf{B}\mathbf{B}^*\mathbf{R}(\iunit\omega)^{-*}|\alpha\delta-\beta\gamma|\mathrm{d}\omega, \label{eq:Pcnew}
\end{align}
where $\mathbf{R}(\cdot)^{-1}=\left(\left(\alpha \cdot+\beta\right)\bfI-(\gamma\cdot + \delta)\bfA\right)^{-1}$. We rewrite  $\mathbf{R}(\cdot)^{-1}$ using the manipulations
\begin{align}
	\mathbf{R}(\cdot)^{-1} &= \left(\cdot(\alpha\bfI- \gamma\bfA)+\beta\bfI -\delta\bfA)\right)^{-1} \nonumber \\
	& = \left(\cdot\bfI-(\beta\bfI-\delta\bfA)(\gamma\bfA-\alpha\bfI)^{-1}\right)^{-1}(\alpha\bfI- \gamma\bfA)^{-1} \nonumber \\
	& = \left(\cdot\bfI - m^{-1}(\bfA)\right)^{-1}(\alpha\bfI- \gamma\bfA)^{-1}. \label{eq:Rinvnew}
\end{align}
Define $\mathbf{Q}_c=|\alpha\delta-\beta\gamma|(\alpha\bfI- \gamma\bfA)^{-1}\mathbf{B} \mathbf{B}^*(\alpha\bfI- \gamma\bfA)^{-*}$. Then,
using~\eqref{eq:Rinvnew} in~\eqref{eq:Pcnew},
the Gramian $\ctrlg$ becomes
\[
	\ctrlg = \frac{1}{2\pi}\int_{-\infty}^{\infty}(\iunit\omega\bfI - m^{-1}(\bfA))^{-1}\mathbf{Q}_c(\iunit\omega\bfI - m^{-1}(\bfA))^{-*}\mathrm{d}\omega.	
\]
Applying Plancherel's theorem results in
\begin{equation}\label{eq:gramianmobius}
	\ctrlg = \int_{0}^{\infty}e^{m^{-1}(\bfA)t}\mathbf{Q}_ce^{m^{-1}(\bfA)^*t}\mathrm{d}t.	
\end{equation}
Since $m\colon\C_-\rightarrow \A$, the eigenvalues of $m^{-1}(\bfA)$ are in the open left-half complex plane so that the exponential term in \eqref{eq:gramianmobius} vanishes for $t\rightarrow\infty$.
As in the standard case, we therefore obtain 
\begin{align*}
	&m^{-1}(\bfA)\ctrlg + \ctrlg m^{-1}(\bfA)^* =
	\\ &\hspace{2cm} =\int_{0}^{\infty}\frac{\mathrm{d}}{\mathrm{d}t}e^{m^{-1}(\bfA)t}\mathbf{Q}_ce^{m^{-1}(\bfA)^*t}\mathrm{d}t = -\mathbf{Q}_c,
\end{align*}
which is the Lyapunov equation \eqref{eq:lyapcon}. Since $m^{-1}\colon\A\rightarrow\C_-$, the uniqueness of $\ctrlg$ follows from \cref{lemma:uniqueness}. Similar arguments apply to prove~\eqref{eq:lyapobs}.
\end{proof}
We now have all the tools to develop the conformal BT algorithm that handles systems with transfer functions of the kind $\bfG\in\newhardy$. 
In \cref{alg:conformalbt} we provide a pseudocode of the proposed algorithm, called \texttt{conformalBT}.
The major difference from classical BT is that the Gramians are defined (via the conformal mapping) with respect to $\mathfrakH_\bfG$ and not $\bfG$ (see \eqref{eq:ctrlgint} and \eqref{eq:obsgint}). However, the resulting projection matrices are then applied directly to the system matrices of the original system $\bfG$ and not $\mathfrakH_\bfG$.
After studying the theoretical properties of \texttt{conformalBT}
in the next section, we will illustrate its performance numerically in
\Cref{sec:5}. We note that
the pseucode first computes the Gramians and then the Cholesky factors in order to keep the presentation to align with the analytical development. In practice one would compute the (approximate) Cholesky factors directly without ever forming  $\ctrlg$ and $\obsg$.
\begin{algorithm}
\begin{algorithmic}[1]
\REQUIRE{FOM $(\bfA,\bfB,\bfC)$, conformal map $\psi$, reduced order $r<n$}
\IF{$\psi$ is a M\"obius transformation as in \cref{th:mobiuslyap}}
\STATE{Solve \eqref{eq:lyapcon} and \eqref{eq:lyapobs} to get $\ctrlg$ and $\obsg$}
\ELSE
\STATE{Compute $\ctrlg$ and $\obsg$ by approximating \eqref{eq:ctrlgint} and \eqref{eq:obsgint}}
\ENDIF
\STATE{Compute Cholesky factorizations $\ctrlg=\mathbf{U}\mathbf{U}^*$ and $\obsg=\mathbf{L}\mathbf{L}^*$}
\STATE{Compute SVD of $\mathbf{U}^*\mathbf{L}$ and partition it as follows 
\[
\mathbf{U}^*\mathbf{L}=\begin{bmatrix}
    \mathbf{Z}_r & \mathbf{Z}_2
\end{bmatrix}
\begin{bmatrix}
    \bfSigma_1 & \\ & \bfSigma_2
\end{bmatrix}
\begin{bmatrix}
    \mathbf{Y}_r^* \\ \mathbf{Y}_2^*
\end{bmatrix}
\]}
\STATE{Compute $\mathbf{W}_r=\mathbf{L}\mathbf{Y}_r\bfSigma_1^{-1/2}$, $\mathbf{V}_r=\mathbf{U}\mathbf{Z}_r\bfSigma_1^{-1/2}$}
\STATE{Compute the reduced system matrices $$\bfA_r= \mathbf{W}_r^*\bfA \mathbf{V}_r, \quad \bfB_r=\mathbf{W}_r^*\bfB, \quad \bfC_r=\bfC\mathbf{V}_r$$}
\RETURN $\bfA_r,\bfB_r,\bfC_r$
\caption{Conformal balanced truncation (\texttt{conformalBT})} \label{alg:conformalbt}
\end{algorithmic}
\end{algorithm}

\section{Stability preservation and $\mathcal{H}_2$ error bound}\label{sec:4}

In this section, we discuss properties of reduced models obtained by \cref{alg:conformalbt}. In particular, for the specific case of a M\"obius transformation, we show preservation of stability and for the general case, we discuss an $\mathcal{H}_2$-type error bound. We begin by relating the range of a specific M\"obius transformation to a Hermitian polynomial.  
\begin{lemma}\label{lemma:polybarA}
    Consider the M\"obius transformation~ \eqref{eq:mobiustrans} such that $m\colon\C_-\rightarrow\A$ with pole on the right half plane. Define the polynomial $h$  as  \begin{equation*}
        h(z) = 
    \begin{bmatrix}
        1 & z 
    \end{bmatrix}
    \begin{bmatrix}
    \beta\alpha^* + \beta^*\alpha & (-\delta\alpha^* - \gamma\beta^*)^*\\
        -\delta\alpha^* - \gamma\beta^* & \delta\gamma^*+\delta^*\gamma 
    \end{bmatrix}
    \begin{bmatrix}
        1 \\ z^* 
    \end{bmatrix}.
    \end{equation*}
    Then it holds that 
    \begin{equation*}
        \mathbb{S} :=  \left\{z\in\C \big| h(z)>0\right\}=\A.
    \end{equation*}
\end{lemma}
\begin{proof}
    ``$\subseteq$''
    Consider $z\in \mathbb{S}$.
    Note that we may also consider $m$ as a bijective mapping from $\C\backslash\{-\tfrac{\delta}{\gamma}\}$ to $ \C\backslash\{\tfrac{\alpha}{\gamma}\}$, see \eqref{eq:mobiustrans}. Let us now first assume that $z\neq \tfrac{\alpha}{\gamma}$. Then there exists $s\in\mathbb{C}\backslash\{-\tfrac{\delta}{\gamma}\}$ such that $z=m(s)$.
    Utilizing the specific form of $m$ in \eqref{eq:mobiustrans} yields
    \begin{equation}\label{eq:hms}
    h(z)=h(m(s)) = -\frac{|\alpha\delta-\beta\gamma|^2}{|\gamma s+\delta|^2}\left(s+s^*\right)>0.
    \end{equation}
    This implies that $s+s^*=2\re\{s\}<0$, i.e., $s\in\C_-$.     On the other hand, for $z=\tfrac{\alpha}{\gamma}$ we have the equality $h(\tfrac{\alpha}{\gamma})=0$, which contradicts $h(z)>0.$
    Hence, it follows that $z=m(s) \in\A$ and therefore $\mathbb{S}\subseteq\A$.\\
    ``$\supseteq$''
    Consider $z\in\A$. Then, since $m$ as a mapping from $\C_-$ to $\A$ is surjective, there exists $s\in \mathbb C_-$ with $m(s)=z.$ 
    Now as in \eqref{eq:hms}  consider the explicit expression for $h(m(s))$. Note that $\gamma s +\delta\neq0$ since $s\in\C_-$ and we assumed $m$ to have its pole in the right half plane. Moreover, since $\re\{s\}<0$ and  $|\alpha\delta-\beta\gamma|^2/|\gamma s+\delta|^2>0$, we have that $h(z)=h(m(s))>0$. This shows ${\mathbb A}\subseteq \mathbb S$; and so $\mathbb{S}=\A$.
\end{proof}
In the generic case of $\A = \C_{-}$, i.e., in the case of asymptotically stable systems with poles in the open left-half plane, BT retains asymptotic stability. The situation is rather different in \texttt{conformalBT} since (i) the balanced system $\mathfrak{H}_
\bfG$~\eqref{eq:frakH} does not have the generic first-order state-space form and (ii) the reduction is applied on the original state-space quantities of $\bfG$  (as in Step 9 of \Cref{alg:conformalbt}) not of $\mathfrak{H}_\bfG$.  When does \texttt{conformalBT}
preserve stability in the sense that the retained poles also lie in the set $\A$? 
 Below we  prove this stability preservation result for \texttt{conformalBT} when specific M\"obius transformations are adopted as conformal maps. Before we state the result, we recall that a \textit{balanced} system has equal and diagonal Gramians. Thus in the setting of 
 \texttt{conformalBT}, $\bfG$ is \emph{conformally balanced} means that the Gramians 
 (of $\frak{H}_\bfG$) defined in \eqref{eq:ctrlgint} and \eqref{eq:obsgint} are equal and diagonal. In addition, when we refer to a \textit{controllable} system we mean that the controllability matrix of $\bfG$ is full rank.

\begin{theorem}\label{th:stabilitypreservation} Let the system $\bfG\in\newhardy$ with poles in the open set $\A$  be controllable. Also let the M\"obius transformation~ \eqref{eq:mobiustrans} parametrized as in \cref{th:mobiuslyap} with a pole in the right half plane or $\gamma=0$, be employed in \textnormal{\texttt{conformalBT}}.
Choose $r$ in Step 7 of \cref{alg:conformalbt} such that $\bfSigma_1$ is positive definite and has no diagonal entries in common with $\bfSigma_2$. Then the reduced system resulting from \textnormal{\texttt{conformalBT}} will have its poles in the open set $\A$.
\end{theorem}
\begin{proof}
Let $\ctrlg$ and $\obsg$ be the solutions to the Lyapunov equations \eqref{eq:lyapcon}  and \eqref{eq:lyapobs}, respectively. 
We start by inserting the formula for the matrix function $m^{-1}(\bfA)=(\beta\bfI - \delta\bfA)(\gamma\bfA-\alpha\bfI)^{-1}$ into~\eqref{eq:lyapcon} to obtain 
    \begin{align*}
        &(\beta\bfI-\bfA\delta)(\gamma\bfA-\alpha\bfI)^{-1}\ctrlg + \ctrlg(\gamma\bfA-\alpha\bfI)^{-*}(\beta\bfI-\bfA\delta)^*\\
        &\hspace{2cm}= -|\alpha\delta-\beta\gamma|(\alpha\bfI- \gamma\bfA)^{-1}\mathbf{B}\mathbf{B}^*(\alpha\bfI- \gamma\bfA)^{-*}.
    \end{align*}
    Using the fact that the two matrices $(\beta\bfI - \delta\bfA)$ and $(\gamma\bfA-\alpha\bfI)^{-1}$ commute, we obtain 
     \begin{align*}
        &-(\beta\bfI - \delta\bfA)\ctrlg(\gamma\bfA-\alpha\bfI)^{*} - (\gamma\bfA-\alpha\bfI)\ctrlg(\beta\bfI - \delta\bfA)^* \\
        &\hspace{2cm}= -|\alpha\delta-\beta\gamma|\bfB\bfB^*,
    \end{align*}  
    which results in
    \begin{equation}\label{eq:mobiuslyapdeveloped}
            \kappa_1\bfA\ctrlg\bfA^* + \kappa_2\ctrlg + \kappa_3\bfA\ctrlg+\kappa_3^*\ctrlg\bfA^* = -|\alpha\delta-\beta\gamma|\bfB\bfB^*,
    \end{equation}
    with 
    \[
    \kappa_1 = -\gamma^*\delta-\gamma\delta^*,\;
    \kappa_2 = -\alpha^*\beta - \alpha\beta^*,\;
    \kappa_3 = \alpha^*\delta + \beta^*\gamma.
    \]
    Without loss of generality we assume $\bfG(\cdot)=\bfC_\mathcal{B}(\cdot\bfI-\bfA_\mathcal{B})^{-1}\bfB_\mathcal{B}$ is (conformally) balanced, i.e., 
    \[
    \ctrlg=\obsg=\bfSigma = 
    \begin{bmatrix}
        \bfSigma_1 & \\ & \bfSigma_2
    \end{bmatrix},
    \]
    with $\mathbf{\Sigma}$ being a diagonal matrix.
    In addition, we partition the balanced system matrices as follows
    \begin{equation}\label{eq:gamma0ornot}
    \bfA_{\mathcal{B}} = \begin{bmatrix}
        \bfA_{r} & \bfA_{12}\\ \bfA_{21} & \bfA_{22}
    \end{bmatrix}, \;
    \bfC_{\mathcal{B}} = \begin{bmatrix}
        \bfC_r & \bfC_{2}
    \end{bmatrix}, \;
    \bfB_{\mathcal{B}} = \begin{bmatrix}
        \bfB_r \\ \bfB_2
    \end{bmatrix}. \;
    \end{equation}
    After substituting the balanced system into \eqref{eq:mobiuslyapdeveloped} we consider the first row and column block equation to obtain
    \begin{align*}&\kappa_1(\bfA_r\bfSigma_1\bfA_r^*+\bfA_{12}\bfSigma_2\bfA^*_{12}) + \kappa_2\bfSigma_1 + \kappa_3\bfA_r\bfSigma_1 + \kappa_3^*\bfSigma_1\bfA_r^*=\\
    &\hspace{5cm}-|\alpha\delta-\beta\gamma| \bfB_r\bfB_r^*.
    \end{align*}
    We now adopt a similar strategy to \cite[Section 7.2.1]{Ant05}. Let $\bfA_r^*\mathbf{v}=\mu \mathbf{v}$ and $\bfA_r\mathbf{x}=\lambda \mathbf{x}$ with $\lambda=\mu^*$ where $\mathbf{v}$ and $\mathbf{x}$ are the left- and right-eigenvectors of $\bfA_r$ corresponding to the eigenvalue 
$\lambda$. Our goal is to show that $\lambda \in \A$. We  multiply the last equation by $\mathbf{v}^*$ and $\mathbf{v}$ from the left and right, respectively to obtain    
    \begin{align*}
        &\kappa_1\left(\mu^*\mathbf{v}^*\bfSigma_1\mathbf{v}\mu + \mathbf{v}^*\bfA_{12}\bfSigma_2\bfA_{12}^*\mathbf{v}\right) + \kappa_2\mathbf{v}^*\bfSigma_1\mathbf{v} + \kappa_3\mu^*\mathbf{v}^* \bfSigma_1\mathbf{v}\\
        & \hspace{3cm}+ \kappa_3^*\mathbf{v}^*\bfSigma_1\mathbf{v}\mu  = -|\alpha\delta-\beta\gamma|\mathbf{v}^*\bfB_r\bfB_r^*\mathbf{v},
    \end{align*}
    which becomes
        \begin{align*}
        &(-\kappa_1|\mu|^2 -  \kappa_3\mu^* - \kappa_3^*\mu - \kappa_2)\mathbf{v}^*\bfSigma_1\mathbf{v} \\
        &\hspace{1cm} =\kappa_1\mathbf{v}^*\bfA_{12}\bfSigma_2\bfA_{12}^*\mathbf{v} + |\alpha\delta-\beta\gamma|\mathbf{v}^*\bfB_r\bfB_r^*\mathbf{v}.
        \end{align*}
    Let us first consider the case in which the M\"obius transformation $m$, given in \eqref{eq:mobiustrans}, has a finite pole in $-\delta/\gamma$, which can be written as
    \begin{equation*}
        -\frac{\delta}{\gamma} = -\frac{\delta\gamma^*}{|\gamma|^2} = -\frac{\re\{\delta\gamma^*\}}{|\gamma|^2} - \iunit\frac{\im\{\delta\gamma^*\}}{|\gamma|^2}.
    \end{equation*}
    We assumed that the pole of $m$ lies in $\C_+$, which translates to having $\re\{\delta\gamma^*\}<0$. This results in 
    \[
    \kappa_1 =  -\delta\gamma^*-\delta^*\gamma = -2\re\{\delta\gamma^*\}>0.
    \]
    Because the terms $\bfA_{12}\bfSigma_2\bfA_{12}^*$ and $\bfB_r\bfB_r^*$ are positive semi-definite we then get 
    \begin{align*}
    &(-\kappa_1|\mu|^2 -  \kappa_3\mu^* - \kappa_3^*\mu - \kappa_2)\mathbf{v}^*\bfSigma_1\mathbf{v} \\
    &\hspace{1cm} =\kappa_1\mathbf{v}^*\bfA_{12}\bfSigma_2\bfA_{12}^*\mathbf{v} + |\alpha\delta-\beta\gamma|\mathbf{v}^*\bfB_r\bfB_r^*\mathbf{v} \geq 0.
    \end{align*}
    The fact that $\bfSigma_1$ is positive definite implies  
    \[
    -\kappa_1|\mu|^2 -  \kappa_3\mu^* - \kappa_3^*\mu - \kappa_2\geq 0
    \]
    After replacing $\mu$ with $\lambda^*$ we get
    \begin{equation}\label{eq:mobiuspreservationstability}
    -\kappa_1|\lambda|^2 -  \kappa_3^*\lambda^* - \kappa_3\lambda - \kappa_2\geq 0.
    \end{equation}
    To prove that the inequality in \eqref{eq:mobiuspreservationstability} is strict
    (so that we can employ \Cref{lemma:polybarA})
    we consider the case 
    \[
    -\kappa_1|\lambda|^2 -  \kappa_3^*\lambda^* - \kappa_3\lambda - \kappa_2= 0.
    \]
    Recall that $\bfA_{12}\bfSigma_2\bfA_{12}^*$ and $\bfB_r\bfB_r^*$ are positive semi-definite, and $\kappa_1>0$. For the equality to hold, we then have $\mathbf{v}^*\bfA_{12}=0$ and $\mathbf{v}^*\bfB_{r}=0$. Because $\mathbf{v}^*$ is also a left eigenvector of $\bfA_{r}$ then $[\mathbf{v}^*\;0]$ is a left eigenvector of $\bfA_{\mathcal{B}}$. Due to $[\mathbf{v}^*\;0]$ being a non-trivial element of the left kernel of $\bfB_\mathcal{B}$ then the controllability matrix
    \[ 
    \begin{bmatrix}
        \bfB_{\mathcal{B}} & \bfA_{\mathcal{B}}\bfB_{\mathcal{B}} & \dots & \bfA_{\mathcal{B}}^{n-1}\bfB_{\mathcal{B}}  
    \end{bmatrix},
    \]
    does not have full rank. This results in the system not being controllable which contradicts the assumption of the theorem.  In other words
    \begin{equation}\label{eq:mobiuspreservationstabilitystrict}
    -\kappa_1|\lambda|^2 -  \kappa_3^*\lambda^* - \kappa_3\lambda - \kappa_2> 0.
    \end{equation}
    Due to \cref{lemma:polybarA} we can then conclude that the ROM poles need to be in $\A$.
    
    The case of $\gamma=0$ needs some special care. First, we note that $m\colon \mathbb C\to \mathbb C$ with $m(s)=\tfrac{\alpha s+\beta}{\delta}$ is bijective and we obtain
    \[
    \kappa_1 = 0,\;
    \kappa_2 = -\alpha^*\beta - \alpha\beta^*,\;
    \kappa_3 = \alpha^*\delta.
    \]
    Then \eqref{eq:mobiuslyapdeveloped} becomes  
    \[
    \kappa_2\ctrlg + \kappa_3\bfA\ctrlg+\kappa_3^*\ctrlg\bfA^* = -|\alpha\delta|\bfB\bfB^*.
    \]
    As above we consider the balanced realization of $\bfG$ with matrices in \eqref{eq:gamma0ornot} resulting in 
    \[
    \kappa_2\bfSigma + \kappa_3\bfA_\mathcal{B}\bfSigma+\kappa_3^*\bfSigma\bfA_\mathcal{B}^* = -|\alpha\delta|\bfB_\mathcal{B}\bfB_\mathcal{B}^*.
    \]
    Expanding the terms leads to 
    \begin{equation}\label{eq:middlegamma0}
    \left(\alpha^*\delta\bfA_\mathcal{B}-\alpha^*\beta\bfI\right)\bfSigma + \bfSigma\left(\alpha^*\delta\bfA_\mathcal{B}-\alpha^*\beta\bfI\right)^* = -|\alpha\delta|\bfB_\mathcal{B}\bfB_\mathcal{B}^*.
    \end{equation}
    By multiplying the left and right hand side of \eqref{eq:middlegamma0} by $|\alpha|^{-2}$ we can rewrite the equality as 
    \[
        \tilde{\bfA}_\mathcal{B}\bfSigma + \bfSigma\tilde{\bfA}_\mathcal{B}^*=-\frac{|\alpha\delta|}{|\alpha|^2} \bfB_\mathcal{B}\bfB_\mathcal{B}^*,
    \]
    where $\tilde{\bfA}_\mathcal{B}=(\delta\bfA_\mathcal{B}-\beta\bfI)\alpha^{-1}=m^{-1}(\bfA_\mathcal{B})$.
    Here, $m^{-1}$ maps the spectrum of $\bfA_\mathcal{B}$ from $\A$ into $\C_-$.
    Since $\gamma=0$, $\tilde{\bfA}_\mathcal{B}$ is a combination of scaling, rotation, and translation of the original matrix  $\bfA_\mathcal{B}$. 
    Consider now the Krylov space spanned by the columns of the controllability matrix
    \begin{equation*}
    \mathcal{K}(\bfB_\mathcal{B},\bfA_{\mathcal{B}})=\textnormal{span}\left\{\bfB_{\mathcal{B}},\; \bfA_{\mathcal{B}}\bfB_{\mathcal{B}},\; \dots ,\;\bfA_{\mathcal{B}}^{n-1}\bfB_{\mathcal{B}}\right\}.
    \end{equation*}
    Due to the system $\bfG$ being controllable, we have that $\mathcal{K}(\bfB_\mathcal{B},\bfA_{\mathcal{B}})$ spans $\C^n$. 
    It is known that a Krylov space is invariant to scaling (including rotation) and translation. For this reason we have that $\mathcal{K}(\frac{\sqrt{\alpha\delta}}{\alpha}\bfB_\mathcal{B},\tilde{\bfA}_{\mathcal{B}})=\mathcal{K}(\bfB_\mathcal{B},\bfA_{\mathcal{B}})$ which results in the system composed of the matrices $\left(\tilde{\bfA}_\mathcal{B},\frac{\sqrt{\alpha\delta}}{\alpha}\bfB_\mathcal{B},\frac{\sqrt{\alpha\delta}}{\alpha}\bfC_\mathcal{B}\right)$ being controllable.
    Now, consider the partition
    \[
    \tilde\bfA_{\mathcal{B}} = \begin{bmatrix}
        \tilde\bfA_{r} & \tilde\bfA_{12}\\ \tilde\bfA_{21} & \tilde\bfA_{22}
    \end{bmatrix},
    \]
    we can then prove that the eigenvalues of $\tilde\bfA_{r}$ lie in $\C_-$ by simply following the proof in \cite[Section 7.2.1]{Ant05} for the continuous-time case with a system having the matrices $\left(\tilde{\bfA}_\mathcal{B},\frac{\sqrt{\alpha\delta}}{\alpha}\bfB_\mathcal{B},\frac{\sqrt{\alpha\delta}}{\alpha}\bfC_\mathcal{B}\right)$. 
    Let $\tilde\lambda$ and $\lambda$  be eigenvalues of $\tilde\bfA_{r}$ and $\bfA_r$, respectively.
    Due to $\re\{\tilde\lambda\}<0$, we then have that
    \begin{equation}\label{eq:inequalitygamma0}
    \re\{\tilde\lambda\} = \re\left\{\frac{\delta\lambda-\beta}{\alpha}\right\} = \re\{m^{-1}(\lambda)\} <0.
    \end{equation} 
    Due to the bijectivity of $m^{-1}$, the only points that satisfy the inequality in \eqref{eq:inequalitygamma0} are $\lambda\in\A$.
\end{proof}
It is important to recall that even if the adopted Gramians are from $\mathfrakH_\bfG$, the theorem applies to the poles of the original system $\bfG$ and not of $\mathfrakH_\bfG$. In \cref{sec:5} we illustrate how \cref{th:stabilitypreservation} applies (and is employed) in two numerical examples.

\begin{remark}\label{remark:1}
    It is important to note that we use the term \textit{stability} in the sense described in \cite[Section 3.4.7]{HinPri05} and not necessarily meaning that the reduced system is asymptotically stable. 
    More precisely, we adopt the term to indicate that if the full order system has its poles in a specific domain $\A$ then the reduced model will have them in the same domain. 
    If asymptotic stability needs to be achieved then it is essential to choose an appropriate conformal map with range in $\C_-$.
\end{remark}

The next theorem provides a bound for the error norm $\|\bfG-\bfG_r\|_{\newhardy}$, which  directly  results from \cite{SorAnt05} where general structured transfer functions are discussed. It is worth mentioning that, in what follows, we do not impose a specific conformal map to be used.
\begin{theorem}\label{th:H2errorbound}
Given $\bfG\in\newhardy$, the reduced system $\bfG_r$  from \textnormal{\texttt{conformalBT}} satisfies the  inequality
\begin{equation}\label{eq:H2bound}
    \|\bfG-\bfG_r\|_{\newhardy}^2\leq \textnormal{trace}\{\mathbf{C}_2\mathbf{\Sigma}_2\mathbf{C}_2^*\} + \varepsilon\mathrm{trace}\{\mathbf{\Sigma}_2\},
\end{equation}
where $\varepsilon\in\R$ is a constant that depends on the (conformally) balanced realization of $\bfG$ and $\bfG_r$, and  $\mathbf{\Sigma}_2$ indicates the neglected singular values in Step 7 of \Cref{alg:conformalbt}. 
\end{theorem}
\begin{proof}
The proof is similar to \cite[Section 4.3.3]{SorAnt05}. Here we highlight the most important steps considering complex valued transfer functions. 
Consider the balanced realization $\mathfrak{H}_\bfG(\cdot)=\bfC_\mathcal{B}\bfK_\mathcal{B}(\cdot)^{-1}\bfB_\mathcal{B}$ and let $\mathbf{N}_\mathcal{B}(\cdot) = \mathbf{K}_\mathcal{B}(\cdot)^{-1}\mathbf{B}_\mathcal{B}$. We then rewrite the controllability Gramian of $\mathfrak{H}_\bfG$ from \eqref{eq:frakH} as
\begin{equation}\label{eq:Xc}
\ctrlg = \frac{1}{2\pi}\int_{-\infty}^{\infty} \mathbf{N}(\mathrm{i}\omega)\mathbf{N}(\mathrm{i} \omega)^*\mathrm{d}\omega=
\underbrace{\begin{bmatrix}
	\mathbf{\Sigma}_1 &  \\  & \mathbf{\Sigma}_2
\end{bmatrix}}_{\mathbf{\Sigma}}.
\end{equation}
We then make the following partitions: 
\begin{equation}\label{eq:balancedsysmatrices}
    \begin{aligned}
    	&\mathbf{K}_\mathcal{B}(\cdot)=\begin{bmatrix}
    		\mathbf{K}_r(\cdot) & \mathbf{K}_{12}(\cdot) \\ \mathbf{K}_{21}(\cdot) & \mathbf{K}_{22}(\cdot)
    	\end{bmatrix},
    	\quad
    	\mathbf{C}_\mathcal{B} = \begin{bmatrix}
    		\mathbf{C}_r & \mathbf{C}_2
    	\end{bmatrix},
    	\\
    	&\mathbf{B}_\mathcal{B} = \begin{bmatrix}
    		\mathbf{B}_r\\ \mathbf{B}_2
    	\end{bmatrix},
    	\quad
    	\mathbf{N}_\mathcal{B}(\cdot) = \begin{bmatrix}
    		\mathbf{N}_{\mathcal{B}1}(\cdot)\\
    		\mathbf{N}_{\mathcal{B}2}(\cdot)
    	\end{bmatrix}.
    \end{aligned}
\end{equation}
By inserting \eqref{eq:balancedsysmatrices} into \eqref{eq:Xc} we obtain
\begin{equation*}
	\begin{aligned}
		\mathbf{\Sigma}_1 &=  \frac{1}{2\pi}\int_{-\infty}^{\infty} \mathbf{N}_{\mathcal{B}1}(\mathrm{i}\omega)\mathbf{N}_{\mathcal{B}1}(\mathrm{i} \omega)^*\mathrm{d}\omega,\\
		\mathbf{\Sigma}_2 &=  \frac{1}{2\pi}\int_{-\infty}^{\infty} \mathbf{N}_{\mathcal{B}2}(\mathrm{i}\omega)\mathbf{N}_{\mathcal{B}2}(\mathrm{i} \omega)^*\mathrm{d}\omega,\\
		0 &=  \frac{1}{2\pi}\int_{-\infty}^{\infty} \mathbf{N}_{\mathcal{B}1}(\mathrm{i}\omega)\mathbf{N}_{\mathcal{B}2}(\mathrm{i} \omega)^*\mathrm{d}\omega.
	\end{aligned}
\end{equation*}
Considering the first row of $\mathbf{B}_\mathcal{B} = \mathbf{K}_\mathcal{B}(\cdot)\mathbf{N}_\mathcal{B}(\cdot)$, we get
$
\bfB_r = \mathbf{K}_r(\cdot)\bfN_{\mathcal{B}1}(\cdot) + \mathbf{K}_{12}(\cdot)\bfN_{\mathcal{B}2}(\cdot)
$.
Define the reduced order quantity $\mathbf{N}_r(\cdot) = \mathbf{K}_r(\cdot)^{-1}\mathbf{B}_r$. 
Then, by substituting the expression for $\bfB_r$ into this definition, 
we obtain 
\begin{equation*}
	\mathbf{N}_r(\cdot) = \mathbf{N}_{\mathcal{B}1}(\cdot) + \mathbf{L}(\cdot)\mathbf{N}_{\mathcal{B}2}(\cdot),
\end{equation*}
with $\mathbf{L}(\cdot)=\mathbf{K}_r(\cdot)^{-1}\mathbf{K}_{12}(\cdot)$.
Consider now the error norm 
\begin{equation}\label{eq:h2equality}
\begin{aligned}
	\|\mathfrak{H}_\mathcal{E}\|_{\mathcal{H}_2}^2 &= \|\mathfrak{H}_\bfG - \mathfrak{H}_{\bfG_r}\|_{\mathcal{H}_2}^2 \\
    &= \|\mathfrak{H}_\bfG\|_{\mathcal{H}_2}^2- 2\textnormal{Re}\left\{\langle\mathfrak{H}_\bfG,\mathfrak{H}_{\bfG_r}\rangle_{\mathcal{H}_2}\right\}+ \|\mathfrak{H}_{\bfG_r}\|_{\mathcal{H}_2}^2,
\end{aligned}
\end{equation}
with the three developed terms being 
\begin{equation}\label{eq:h2equality1}
	\|\mathfrak{H}_\bfG\|_{\mathcal{H}_2}^2 = \textnormal{trace}\{\mathbf{C}_r\mathbf{\Sigma}_1\mathbf{C}_r^*\} + \textnormal{trace}\{\mathbf{C}_2\mathbf{\Sigma}_2\mathbf{C}_2^*\},
\end{equation}
\begin{equation}\label{eq:h2equality2}
\begin{aligned}
    \langle\mathfrak{H}_\bfG,\mathfrak{H}_{\bfG_r}\rangle_{\mathcal{H}_2} &= \textnormal{trace}\{\mathbf{C}_r\mathbf{\Sigma}_1\mathbf{C}_r^*\} \\
    &\hspace{-1cm}+ \frac{1}{2\pi}\int_{-\infty}^{\infty} \textnormal{trace}\{\mathbf{C}_\mathcal{B}\mathbf{N}_\mathcal{B}(\mathrm{i}\omega)\mathbf{N}_{\mathcal{B}2}(\mathrm{i} \omega)^*\mathbf{L}(\mathrm{i}\omega)^*\mathbf{C}_r^*\}\mathrm{d}\omega,\\
\end{aligned}
\end{equation}
\begin{equation}\label{eq:h2equality3}
\begin{aligned}
    \|\mathfrak{H}_{\bfG_r}\|_{\mathcal{H}_2}^2&= \textnormal{trace}\{\mathbf{C}_r\mathbf{\Sigma}_1\mathbf{C}_r^*\} \\
    &\hspace{-2cm}+ 2\textnormal{Re}\left\{\frac{1}{2\pi}\int_{-\infty}^{\infty} \textnormal{trace}\{\mathbf{C}_r\mathbf{N}_{\mathcal{B}1}(\mathrm{i}\omega)\mathbf{N}_{\mathcal{B}2}(\mathrm{i} \omega)^*\mathbf{L}(\mathrm{i}\omega)^*\mathbf{C}_r^*\}\mathrm{d}\omega\right\}\\
	&\hspace{-2cm}+\frac{1}{2\pi}\int_{-\infty}^{\infty} \textnormal{trace}\{\mathbf{C}_r\mathbf{L}(\mathrm{i}\omega)\mathbf{N}_{\mathcal{B}2}(\mathrm{i} \omega)\mathbf{N}_{\mathcal{B}2}(\mathrm{i} \omega)^*\mathbf{L}(\mathrm{i}\omega)^*\mathbf{C}_r^*\}\mathrm{d}\omega.
\end{aligned}
\end{equation}
By plugging \eqref{eq:h2equality1}, \eqref{eq:h2equality2}, and \eqref{eq:h2equality3} into \eqref{eq:h2equality} we get
\begin{align*}\|\mathfrak{H}_\mathcal{E}\|_{\mathcal{H}_2}^2 &= \textnormal{trace}\{\mathbf{C}_2\mathbf{\Sigma}_2\mathbf{C}_2^*\} \\
    &+\textnormal{Re}\Big\{\frac{1}{2\pi}\int_{-\infty}^{\infty} \textnormal{trace}\{\left(\mathbf{C}_r\mathbf{L}(\mathrm{i}\omega)-2\mathbf{C}_2\right)\\
    &\hspace{2cm}\mathbf{N}_{\mathcal{B}2}(\mathrm{i}\omega)\mathbf{N}_{\mathcal{B}2}(\mathrm{i} \omega)^*\mathbf{L}(\mathrm{i}\omega)^*\mathbf{C}_r^*\}\mathrm{d}\omega\Big\}\\
\end{align*}
with the bound
\begin{align*}
	\|\mathfrak{H}_\mathcal{E}\|_{\mathcal{H}_2}^2&\leq \textnormal{trace}\{\mathbf{C}_2\mathbf{\Sigma}_2\mathbf{C}_2^*\} \\
    &\hspace{0.5cm}+ \sup_{\omega}\|\mathbf{L}(\mathrm{i}\omega)^*\mathbf{C}_r^*(\mathbf{C}_r\mathbf{L}(\mathrm{i}\omega)-2\mathbf{C}_2)\|_2\textnormal{trace}\{\mathbf{\Sigma}_2\}.	
\end{align*}
\end{proof}

\section{Numerical experiments} \label{sec:5}
In this section, we test \texttt{conformalBT}
on three numerical examples. 
While the spectrum of the full order system is in $\C_-$ for the first example, it is in $\mathrm{i}\R_+$ for the second, and in $\mathrm{i}\R$ for the third.
All the numerical experiments were generated on a Lenovo ThinkPad with an 8 core Intel\textsuperscript{\textregistered} i7-8565U 1.8GHz processor, 48GB of RAM, and MATLAB R2023b.
For all PDE examples, we used a spatial semi-discretization by centered finite differences. 

To compute the conformal Gramians while the \texttt{lyapchol} command was used for the first two examples, in \cref{sec:heat} and \cref{sec:sch}, an adaptive Gauss-Kronrod algorithm was used for 
the third example in~\cref{sec:wave} and its implementation follows the steps described in \cite{Bre16}. For the computation of the $\hardy_2$ norms and the output trajectories we relied on the commands \texttt{integral} and \texttt{ode23}, respectively. Both the \texttt{ode23} and the \texttt{integral} functions adopt relative and absolute tolerances of $10^{-8}$ and $10^{-12}$, respectively. The code to generate the numerical results is  available at \url{https://github.com/aaborghi/conformalBT.git}

\subsection{Heat equation}\label{sec:heat}
In the first example we consider the boundary controlled heat equation described as
\begin{equation*}
	\begin{aligned}
		\frac{\partial w(x,t)}{\partial t}&=\frac{\partial^2w(x,t)}{\partial x^2}, && \text{on } (0,1)\times(0,T), \\
		w(0,t) &= 0, \; w(1,t) = u(t), && \text{on } (0,T), \\
		y(t) &= \int_{0.1}^{0.4} w(x,t)\mathrm{d}x, && \text{on } (0,T),\\
		w(x,0)&=0, && \text{in } (0,1),
	\end{aligned}
\end{equation*}
where $u$ and $y$ are, respectively, the scalar input and output. 
The discretization results in a full order system of dimension $n = 200$, $q=m=1$, and poles in the negative real axis $\R_-$. This example was specifically chosen due to its spectrum being in the left-half plane so that we can compare the performance of \texttt{conformalBT} with that of classical balanced truncation. 
For the former we choose a M\"obius transformation~\eqref{eq:mobiustrans} that maps the left-half plane $\C_-$ into the disk $\A=\D_{R,c}=\{z\in\C\big| |z-c|< R\}$ and $\iunit\R$ to $\partial\D_{R,c}$. To do so, we choose the parameters $\alpha,\beta,\gamma,\delta$ such that we have 
\begin{equation}\label{eq:mobiusheat}
\psi(\cdot) = c + R\frac{\cdot +1}{\cdot - 1}.
\end{equation}
Because $\psi\colon\C_-\rightarrow\D_{R,c}$ is a mapping that satisfies the assumptions of \cref{th:stabilitypreservation} we then have that \texttt{conformalBT} computes a ROM with poles in $\D_{R,c}$. As a matter of fact, plugging in the values of $\alpha,\beta,\gamma$ and $\delta$ of \eqref{eq:mobiusheat} in the inequality \eqref{eq:mobiuspreservationstabilitystrict} results in
    \[
    |\lambda-c|^2 < R^2,
    \]
which indicates the disk $\D_{R,c}$. 
A similar result can be found in \cite[Example 3.4.104]{HinPri05} by replacing $\alpha, \beta,\gamma, \delta$ in \eqref{eq:mobiuslyapdeveloped} with $R+c, R-c, 1, -1$, respectively.
To enclose the spectrum of the system we choose $c=-17\times 10^4$ and $R=17\times 10^4$.
Recalling \cref{remark:1} we need the range of \eqref{eq:mobiusheat} to be a subset of $\C_-$ so to avoid the potential placement of the ROM poles in the right half plane. For this reason we chose $R=-c$.
Since \eqref{eq:mobiusheat} satisfies the assumptions in \cref{th:mobiuslyap}, we compute the Gramians by solving the Lyapunov equations \eqref{eq:lyapcon} and \eqref{eq:lyapobs}.

We present the $\newhardy$ error norm along with the error bound in the top plot of \cref{fig:heatdyn} for various $r$ values. In addition, the bottom plot shows the output relative error of the impulse response for a reduced model with $r=10$ computed by \texttt{conformalBT} and BT. The result shows that 
the two algorithms reach the same level of accuracy. 
\begin{figure}[htb]
\centering
\input{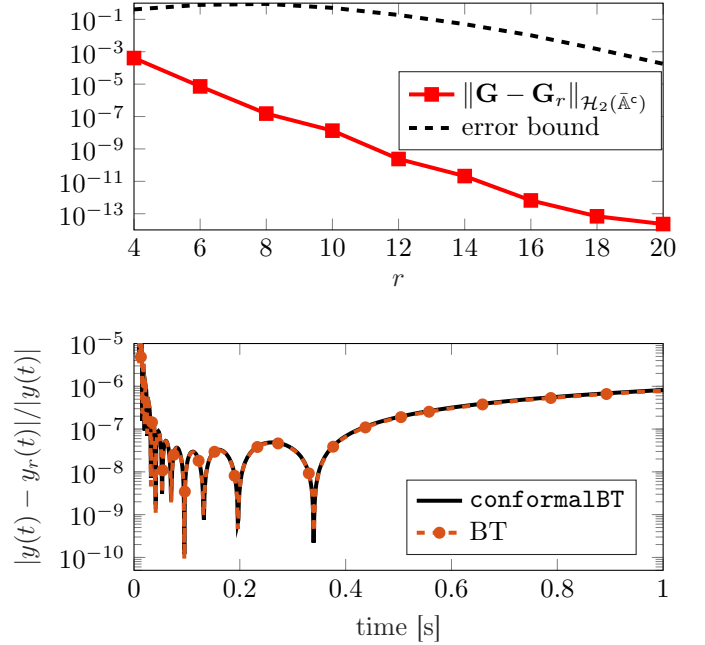}
\caption{(Top) $\newhardy$ error norm between the full and reduced order systems along with the $\hardy_2$ error bound in \eqref{eq:H2bound} for different values of $r$. (Bottom) relative error of the ROM impulse response with $r=10$ computed with \texttt{conformalBT} and BT applied to the Heat equation with $n=200$.}
\label{fig:heatdyn}
\end{figure}

\subsection{Schr\"odinger equation}\label{sec:sch}
In this example we test \texttt{conformalBT} as given in \cref{alg:conformalbt} on a controlled variant of the Schr\"odinger equation
\begin{equation*}
	\begin{aligned}
		\frac{\partial w(x,t)}{\partial t}&=-\mathrm{i} \frac{\partial^2w(x,t)}{\partial x^2}+\chi_{[0.4,0.5]}u^{(1)}(t)\\
  & \quad+\chi_{[0.5,0.6]}u^{(2)}(t), && \text{on } (0,1)\times(0,T), \\
		w(0,t)&=0, \ \ w(1,t)=0, && \text{on } (0,T), \\
		\mathbf{y}(t) &= \begin{bmatrix} \vspace{0.2cm}
		    \int_{0.1}^{0.3} w(x,t)\,\mathrm{d}x\\ 
            \int_{0.7}^{0.9} w(x,t)\,\mathrm{d}x\\
		\end{bmatrix}, && \text{on } (0,T),\\
		w(x,0)&=0, && \text{in } (0,1),
	\end{aligned}
\end{equation*}
where $u^{(1)}$ and $u^{(2)}$, and $\mathbf{y}$ are the inputs and outputs, respectively. 
After discretization, the differential equation results in a FOM with $q=2$, $m=2$, and $n=1000$. Because the spectrum of the system is on the upper part of the imaginary axis we choose the clockwise rotatory conformal map
\begin{equation}\label{eq:mobiusschroedinger}
\psi(\cdot) = -\iunit\cdot
\end{equation}
with $\psi\colon\C_-\rightarrow\C_{\uparrow}$, and $\psi\colon\iunit\R\rightarrow\R$, where $\A=\C_{\uparrow}$ is the open upper half complex plane. Similarly to \cref{sec:heat}, also here the conformal map chosen in \eqref{eq:mobiusschroedinger} satisfies the assumptions in \cref{th:stabilitypreservation}. As a matter of fact, plugging in the coefficients of the M\"obius transformation in \eqref{eq:mobiusschroedinger} into \eqref{eq:mobiuspreservationstabilitystrict} results in the inequality $\im(\lambda)>0$, which exactly defines $\C_{\uparrow}$. In addition, \eqref{eq:mobiusschroedinger} satisfies also the assumptions in \cref{th:mobiuslyap}, allowing us to compute the Gramians by solving \eqref{eq:lyapcon} and \eqref{eq:lyapobs}.

In the top plot of \cref{fig:schdyn} we show the evolution of the $\newhardy$  error norm between the FOM and the ROM computed with \texttt{conformalBT} for different reduced orders $r$ along with the error bound given in \cref{th:H2errorbound}. The middle plot shows the output relative error of the step response with $r=9$. 
The spikes in the error are due to all the output signals reaching $0$. It can be seen from the error plot that the resulting ROM output well approximates the FOM response given the control inputs $u^{(1)} = u^{(2)} = \mathbf{u}$ showed in the bottom plot.

We note that since $\psi$ satisfies the assumptions of \cref{th:stabilitypreservation},
it is assured that the reduced system poles are in $\A=\C_{\uparrow}$, thus the reduced system retains stability in the sense of  \cref{th:stabilitypreservation}. However, this does not mean that the reduced poles will also exactly lie on the upper part of the imaginary axis since the upper half plane includes the upper-right and upper-left quadrants. This can result in a ROM that does not mirror the output behaviour of the full order system. In the next example we choose a conformal map designed to overcome this issue.

\begin{figure}[htb]
\centering\input{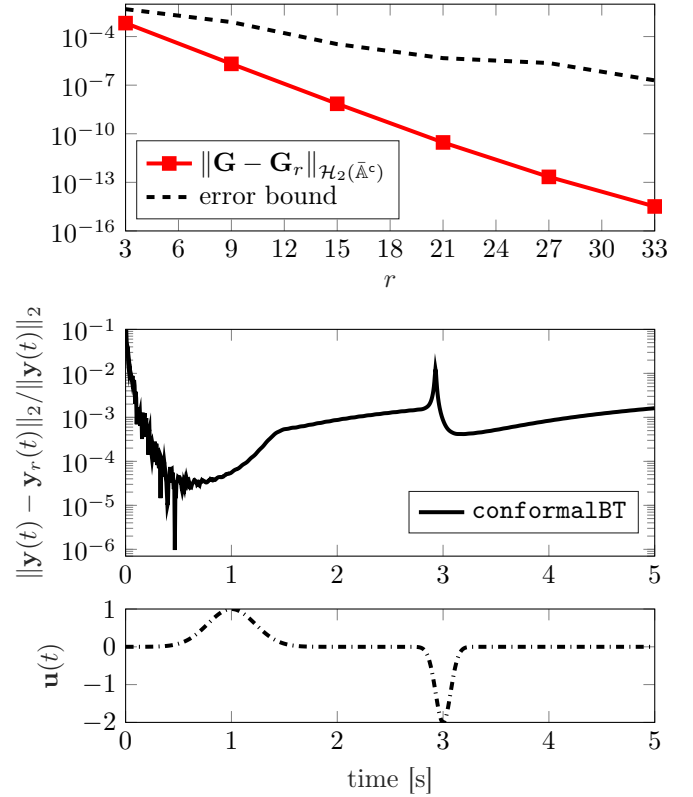}
\caption{(Top) $\newhardy$ error norm between the full and reduced order systems along with the $\hardy_2$ error bound in \eqref{eq:H2bound} for different values of $r$. (Middle) relative error of the ROM step response with $r=9$ computed with \texttt{conformalBT} applied to the Schr\"odinger equation with $n=1000$. (Bottom) Input adopted for evaluating \texttt{conformalBT}. Here, $u^{(1)}$ and $u^{(2)}$ follow the same trajectory equal to $\mathbf{u}$.}
\label{fig:schdyn}
\end{figure}

\subsection{Undamped linear wave equation}\label{sec:wave}
The last example consists of the controlled wave equation
\begin{equation*}
	\begin{aligned}
	\frac{\partial^2w(x,t)}{\partial t^2}&= \frac{\partial^2w(x,t)}{\partial x^2}+ \chi_{[0.1,0.2]}u^{(1)}(t)\\
    &\quad + \chi_{[0.8,0.9]}u^{(2)}(t), && \text{on } (0,1)\times(0,T), \\
		w(0,t)&=0, \ \ w(1,t)=0, && \text{on } (0,T), \\
		\mathbf{y}(t) &= \begin{bmatrix} \vspace{0.2cm}
		    \int_{0.3}^{0.5} w(x,t)\,\mathrm{d}x\\ \vspace{0.2cm}
            \int_{0.6}^{0.7} w(x,t)\,\mathrm{d}x\\ 
		\end{bmatrix}, && \text{on } (0,T),\\
		w(x,0)&=0, && \text{in } (0,1).
	\end{aligned}
\end{equation*}
After discretization we get a full order model with $n=5000$, $q=2$, and $m=2$. 
The spectrum of this system lies on the imaginary axis. 
We designed a function $\psi$ based on the Joukowski transform that maps part of the left-half plane, including the imaginary axis, into a Bernstein ellipse $\mathbb{B}$ excluding the strip $[-1,1]$, i.e., $\mathbb{B}\backslash[-1,1]$. 
This is then translated by $c\in\C$ and scaled by $M\in\C$. We refer to the resulting domain with and without the strip as $\mathbb{B}_{M,c}$ and $\A=\tilde{\mathbb{B}}_{M,c}$, respectively. A thorough analysis of this function is given in \cite[Section 4.1.3]{BorBre23}. The corresponding conformal map is defined as
\begin{equation}\label{eq:joukowskiwave}
\psi(\cdot)=c+\frac{M}{2}\left(R\frac{\cdot+1}{\cdot-1} + \frac{1}{R}\frac{\cdot-1}{\cdot+1}\right)
\end{equation}
with $\psi\colon\mathbb{X}\rightarrow\tilde{\mathbb{B}}_{M,c}$ and $\psi\colon\iunit\R\rightarrow\partial\mathbb{B}_{M,c}$. Here, $\mathbb{X}$ is a subset of $\C_-$ that includes the imaginary axis and $\partial\mathbb{B}_{M,c}$ coincides with $\partial\A^+$ in \cref{assumption:1}.3. A graphical representation of \eqref{eq:joukowskiwave} is showed in \cref{fig:joukowskiwave}.
To avoid computing reduced systems with poles off the imaginary axis, as discussed in \cref{sec:sch}, we choose the parameters of \eqref{eq:joukowskiwave} such that the ellipse $\partial\mathbb{B}_{M,c}$ has the minor semi-axis that is small enough for the set $\tilde{\mathbb{B}}_{M,c}$ to cover the section of the imaginary axis with the FOM poles as closely as possible. 
Since \eqref{eq:joukowskiwave} does not satisfy the assumptions of \cref{th:mobiuslyap}, we compute the Gramians by approximating \eqref{eq:ctrlgint} and \eqref{eq:obsgint} using the adaptive Gauss-Kronrod quadrature (see \cite[Section 5.1]{Bre16} and citations therein). It is interesting to point out that, even if not proven, in this numerical example the poles of $\mathbf{G}_r$ computed with \cref{alg:conformalbt} lie inside $\tilde{\mathbb{B}}_{M,c}$. By keeping the minor semi-axis as small as possible we are able to keep the reduced model poles approximately on the imaginary axis. This is done by setting the value of $R$ close to $1$. Accordingly, we choose the following parameters for \eqref{eq:joukowskiwave}: $R=1+10^{-5}$, $M=10^4$, $c=10^{-6}$. The disadvantage of decreasing the minor semi-axis is that the ellipse contour approaches the spectrum of the full order model. This makes the offline computation of the Gramians computationally expensive as the adaptive Gauss-Kronrod needs to refine the integration interval several times. Nevertheless, the resulting ROM computed by \texttt{conformalBT} with $r=40$ have its poles almost nearly on the imaginary axis with a maximum real part of approximately $\pm 10^{-10}$,  thus fulfilling our objective. \cref{fig:expwave} shows the impulse response of the discretized wave equation compared to the reduced model computed by \texttt{conformalBT}. The error in the bottom plot of \cref{fig:expwave} shows that the resulting ROM output provides a high-fidelity approximation to the FOM output dynamics.

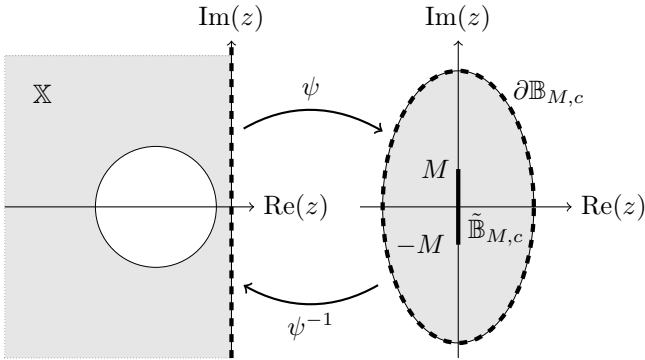
\begin{figure}[htb]
    \centering
    \begin{tikzpicture}[scale=1]

        \draw[ultra thin, densely dotted, fill=gray!20] (-3,-2) rectangle (0,2);
        \draw[->] (0,-2) -- (0,2.2) node[above] {$\text{Im}(z)$};
        \draw[ultra thick, dashed , black] (0,-2) -- (0,2.2);
        \draw [ultra thin, fill=white!100] (-1,0) circle (0.8);
        \draw[->] (-3,0) -- (0.3,0) node[right] {$\text{Re}(z)$};
        \node at (-2.5,1.5) {$\mathbb{X}$};

        \draw [ultra thin, fill=gray!20] (5-2,0) ellipse (1cm and 1.8cm);
        \draw [black, dashed, ultra thick] (5-2,0) ellipse (1cm and 1.8cm);
        \draw[->] (3.2-1.5,0) -- (7-2.5,0) node[right] {$\text{Re}(z)$};
        \draw[->] (5-2,-2) -- (5-2,2.2) node[above] {$\text{Im}(z)$};
        \node at (5.5-2,-0.3) {$\tilde{\mathbb{B}}_{M,c}$};
        \node at (6.2-2,1.5) {$\partial\mathbb{B}_{M,c}$};
        \draw[ultra thick] (5-2,0.5) -- (5-2,-0.5);
        \node at (4.5-2,-0.5) {$-M$};
        \node at (4.7-2,0.5) {$M$};

        \draw[->, thick, shorten >=0pt, shorten <=2pt] (0.1,1) to[bend left] node[midway, above] {$\psi$} (4-2,1);
        \draw[->, thick, shorten >=2pt, shorten <=2pt] (4-2,-1) to[bend left] node[midway, below] {$\psi^{-1}$} (0.1,-1);
        
    \end{tikzpicture}
    \caption{A depiction of the conformal map in \eqref{eq:joukowskiwave} centered at the origin ($c=0$). The grey sets on the left and on the right are, respectively, the domain and range of $\psi$. Here we have the scaling being $M\in\iunit\R$. The thick line in $\tilde{\mathbb{B}}_{M,c}$ indicates the strip $[-1,1]$ after the scaling.}
    \label{fig:joukowskiwave}
\end{figure}

\begin{figure}[htb]
    \centering
    \input{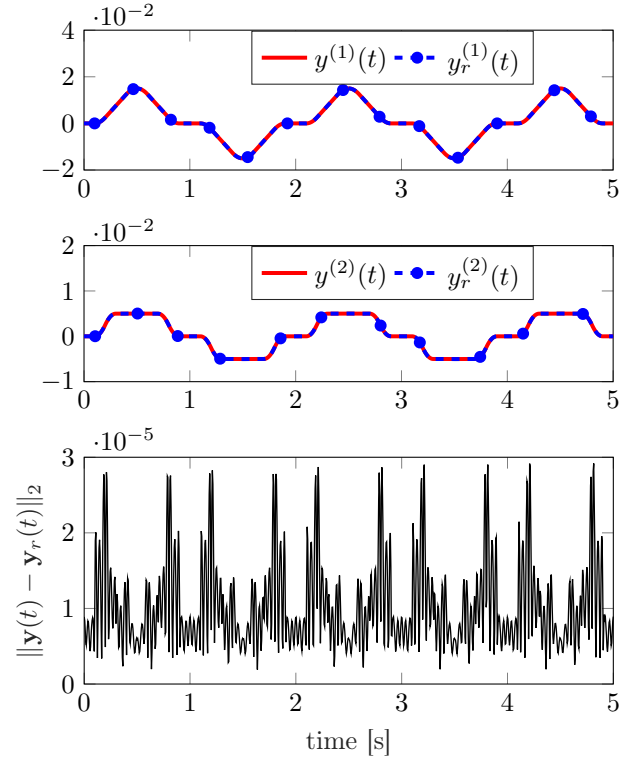}
    \caption{(Top) impulse response of the full and reduced order systems, $y^{(i)}$ and $y_r^{(i)}$, respectively, with $i=1,2$. Here \texttt{conformalBT} computed an $r=40$ reduced order model on a discretized wave equation with $n=5000$.  (Bottom) the output error.}
    \label{fig:expwave}
\end{figure}

\section{Conclusions}
In this paper, we presented a new balanced truncation framework which allows the treatment of transfer functions with poles in general domains. We adopted conformal maps and the $\newhardy$ space to define the Gramians related to these particular systems. We showed that when the M\"obius transformation is used as conformal map, it is possible to compute the new Gramians by solving modified Lyapunov equations. For the proposed algorithm \texttt{conformalBT} we proved that the resulting reduced model has a bounded $\hardy_2$ error norm and that, when the M\"obius transformation is adopted, it preserves the stability of the original full order system.

\subsection*{Acknowledgments}
We would like to thank Jesper Schr\"oder for his helpful discussions.
The work of Borghi and Breiten was funded by the Deutsche Forschungsgemeinschaft (DFG, German Research Foundation) - 384950143 as part of GRK2433 DAEDALUS.  
Gugercin's work was
supported in part by the US National Science Foundation under Grant CMMI-2130695.

\subsection*{Conflict of interest}
The authors declare no competing interests.

\addcontentsline{toc}{section}{References}
\bibliographystyle{spmpsci}
\bibliography{myref}
\end{document}